\documentclass[reqno,a4paper]{amsart}
\usepackage{amsfonts}

\newtheorem{theorem}{Theorem}
\theoremstyle{plain}

\newtheorem{corollary}{Corollary}

\newtheorem{example}{Example}

\newtheorem{remark}{Remark}

\numberwithin{equation}{section}
\input{tcilatex}

\begin{document}
\title[Integral Inequalities for $K_{s}^{2}$]{On Some New Integral
Inequalities for $K_{s}^{2}$}
\author{Mevl\"{u}t TUN\c{C}}
\address{University of Kilis 7 Aralik, Faculty of Arts and Sciences,
Department of Mathematics, 79000, Kilis, Turkey}

\begin{abstract}
In this paper we establish some new inequalities of Hadamard-type for
product of convex and $s-$convex functions in the second sense.
\end{abstract}

\email{mevluttunc@kilis.edu.tr}
\maketitle

\section{Introduction}

A largely applied inequality for convex functions, due to its geometrical
significance, is Hadamard's inequality (see $\left[ \ref{ssd}\right] ,\left[ %
\ref{had}\right] $ or $\left[ \ref{pach}\right] $) which has generated a
wide range of directions for extension and a rich mathematical literature.
The following definitions are well known in the mathematical literature: a
function $f:I\rightarrow 
\mathbb{R}
,\emptyset \neq I\subseteq 
\mathbb{R}
,$ is said to be convex on $I$ if inequality

\begin{equation}
f\left( tx+\left( 1-t\right) y\right) \leq tf\left( x\right) +\left(
1-t\right) f\left( y\right)  \label{1.1}
\end{equation}%
holds for all $x,y\in I$ and $t\in \left[ 0,1\right] $. Geometrically, this
means that if $P,Q$ and $R$ are three distinct points on the graph of $f$
with $Q$ between $P$ and $R$, then $Q$ is on or below chord $PR$.

In the paper $[\ref{hudz}]$ Hudzik and Maligranda considered, among others,
the class of functions which are $s-$convex in the second sense. This class
is defined in the following way: [\ref{breck}] A function $f:\left[ 0,\infty
\right) \rightarrow 
\mathbb{R}
$ is said to be $s-$convex in the second sense if%
\begin{equation}
f\left( tx+\left( 1-t\right) y\right) \leq t^{s}f\left( x\right) +\left(
1-t\right) ^{s}f\left( y\right)  \label{1.2}
\end{equation}%
holds for all $x,y\in \left[ 0,\infty \right) ,$ $t\in \left[ 0,1\right] $
and for some fixed $s\in \left( 0,1\right] $. The class of $s-$convex
functions in the second sense is usually denoted with $K_{s}^{2}$.

It can be easily seen that for $s=1,$ $s-$convexity reduces to ordinary
convexity of functions defined on $\left[ 0,\infty \right) $.

In the same paper $[\ref{hudz}]$ Hudzik and Maligranda proved that if $s\in
\left( 0,1\right) ,$ $f\in K_{s}^{2}$ implies $f\left( \left[ 0,\infty
\right) \right) \subseteq \left[ 0,\infty \right) $,i.e., they proved that
all functions from $K_{s}^{2}$, $s\in \left( 0,1\right) ,$ are nonnegative.

\begin{example}
\label{e1}[\ref{hudz}]. Let $s\in \left( 0,1\right) $ and $a,b,c\in 
\mathbb{R}
$. We define function $f:\left[ 0,\infty \right) \rightarrow 
\mathbb{R}
$ as%
\begin{equation}
f\left( t\right) =\left\{ 
\begin{array}{l}
{a,}\text{ \ \ \ \ \ \ \ \ \ \ \ \ \ }{t=0,} \\ 
bt^{s}+c,\text{ \ \ \ \ \ }t>0.%
\end{array}%
\right.  \label{1.4}
\end{equation}%
It can be easily checked that

(1) If $b\geq 0$ and $0\leq c\leq a$, then $f\in K_{s}^{2}$

(2) If $b>0$ and $c<0$, then $f\notin K_{s}^{2}$
\end{example}

Many important inequalities are established for the class of convex
functions, but one of the most famous is so called Hermite-Hadamard
inequality (or Hadamard's inequality). This double inequality is stated as
follows (see for example [\ref{ppt}, p.137]): let $f$ be a convex function
on $\left[ a,b\right] \subset 
\mathbb{R}
$, where $a\neq b$. Then%
\begin{equation}
f\left( \frac{a+b}{2}\right) \leq \frac{1}{b-a}\int_{a}^{b}f\left( x\right)
dx\leq \frac{f\left( a\right) +f\left( b\right) }{2}.  \label{1.5}
\end{equation}

In the paper [\ref{tunc}] Tun\c{c} established one new Hadamard-type
inequality for products of convex functions. It is given in the next theorem.

\begin{theorem}
\label{t1}[\ref{tunc}]Let $f,g:\left[ a,b\right] \rightarrow 
\mathbb{R}
$ be two convex functions and $fg\in L^{1}\left( \left[ a,b\right] \right) $%
. Then, 
\begin{eqnarray}
&&\frac{1}{\left( b-a\right) ^{2}}\int_{a}^{b}\left( b-x\right) \left(
f\left( a\right) g\left( x\right) +g\left( a\right) f\left( x\right) \right)
dx  \notag \\
&&+\frac{1}{\left( b-a\right) ^{2}}\int_{a}^{b}\left( x-a\right) \left(
f\left( b\right) g\left( x\right) +g\left( b\right) f\left( x\right) \right)
dx  \label{1.6} \\
&\leq &\frac{1}{b-a}\int_{a}^{b}f\left( x\right) g\left( x\right) dx+\frac{%
M\left( a,b\right) }{3}+\frac{N\left( a,b\right) }{6},  \notag
\end{eqnarray}%
where $M(a,b)=f(a)g(a)+f(b)g(b),$ $N(a,b)=f(a)g(b)+f(b)g(a).$
\end{theorem}

\bigskip \bigskip The main purpose of this paper is to establish new
inequalities as given in Theorem \ref{t1}, but now for the class of $s-$%
convex functions in the second sense by using the elementary inequalities.

\section{Main Results}

In the our next theorems we will also make use of Beta function of Euler
type, which is for $u,v>0$ defined as%
\begin{equation*}
\beta \left( u,v\right) =\int_{0}^{1}t^{u-1}\left( 1-t\right) ^{v-1}dt=\frac{%
\Gamma \left( u\right) \Gamma \left( v\right) }{\Gamma \left( u+v\right) }
\end{equation*}%
and%
\begin{equation*}
\beta \left( u,v\right) =\beta \left( v,u\right) ,
\end{equation*}%
where the gamma function, denoted by $\Gamma \left( x\right) $, provides a
generalization of factorial n to the case in which $n$ is not an integer.

\begin{theorem}
\label{t2}\bigskip Let $f,g:I\rightarrow 
\mathbb{R}
,$ $I\subset \left[ 0,\infty \right) $, $a,b\in I,$ with $a<b$\ be functions
such that $f,g$ and $fg$ are in $L^{1}\left( \left[ a,b\right] \right) .$ $f$
is convex and $g$ is $s-$convex function in the second sense on $\left[ a,b%
\right] ,$ for some $s\in \left( 0,1\right] ,$ then 
\begin{eqnarray}
&&\frac{f\left( a\right) }{\left( b-a\right) ^{2}}\int_{a}^{b}\left(
b-x\right) g\left( x\right) dx+\frac{f\left( b\right) }{\left( b-a\right)
^{2}}\int_{a}^{b}\left( x-a\right) g\left( x\right) dx  \notag \\
&&+\frac{g\left( a\right) }{\left( b-a\right) ^{s+1}}\int_{a}^{b}\left(
b-x\right) ^{s}f\left( x\right) dx+\frac{g\left( b\right) }{\left(
b-a\right) ^{s+1}}\int_{a}^{b}\left( x-a\right) ^{s}f\left( x\right) dx 
\notag \\
&\leq &\frac{1}{b-a}\int_{a}^{b}f\left( x\right) g\left( x\right) dx+\frac{%
M\left( a,b\right) }{s+2}+\frac{N\left( a,b\right) }{\left( s+1\right)
\left( s+2\right) }  \label{2.5}
\end{eqnarray}%
where $M\left( a,b\right) =f\left( a\right) g\left( a\right) +f\left(
b\right) g\left( b\right) $ and $N\left( a,b\right) =f\left( a\right)
g\left( b\right) +f\left( b\right) g\left( a\right) $.
\end{theorem}

\begin{proof}
Since $f$ is convex and $g$ is $s-$convex on $\left[ a,b\right] ,$ we have%
\begin{eqnarray*}
f\left( ta+\left( 1-t\right) b\right) &\leq &tf\left( a\right) +\left(
1-t\right) f\left( b\right) \\
g\left( ta+\left( 1-t\right) b\right) &\leq &t^{s}g\left( a\right) +\left(
1-t\right) ^{s}g\left( b\right)
\end{eqnarray*}%
for all $t\in \left[ 0,1\right] .$ Now, using the elementary inequality [\ref%
{ort}, p.4] $\left( a-b\right) \left( c-d\right) \geq 0$ $\left( a,b,c,d\in 
\mathbb{R}
\text{ and }a<b,c<d\right) ,$ we get inequality:%
\begin{eqnarray*}
&&tf\left( a\right) g\left( ta+\left( 1-t\right) b\right) +\left( 1-t\right)
f\left( b\right) g\left( ta+\left( 1-t\right) b\right) \\
&&+t^{s}g\left( a\right) f\left( ta+\left( 1-t\right) b\right) +\left(
1-t\right) ^{s}g\left( b\right) f\left( ta+\left( 1-t\right) b\right) \\
&\leq &f\left( ta+\left( 1-t\right) b\right) g\left( ta+\left( 1-t\right)
b\right) +t^{s+1}f\left( a\right) g\left( a\right) \\
&&+t\left( 1-t\right) ^{s}f\left( a\right) g\left( b\right) +t^{s}\left(
1-t\right) f\left( b\right) g\left( a\right) \\
&&+\left( 1-t\right) ^{s+1}f\left( b\right) g\left( b\right)
\end{eqnarray*}%
Integrating this inequality over $t$ on $\left[ 0,1\right] $, we deduce that%
\begin{eqnarray*}
&&f\left( a\right) \int_{0}^{1}tg\left( ta+\left( 1-t\right) b\right)
dt+f\left( b\right) \int_{0}^{1}\left( 1-t\right) g\left( ta+\left(
1-t\right) b\right) dt \\
&&+g\left( a\right) \int_{0}^{1}t^{s}f\left( ta+\left( 1-t\right) b\right)
dt+g\left( b\right) \int_{0}^{1}\left( 1-t\right) ^{s}f\left( ta+\left(
1-t\right) b\right) dt \\
&\leq &\int_{0}^{1}f\left( ta+\left( 1-t\right) b\right) g\left( ta+\left(
1-t\right) b\right) dt \\
&&+f\left( a\right) g\left( a\right) \int_{0}^{1}t^{s+1}dt+f\left( a\right)
g\left( b\right) \int_{0}^{1}t\left( 1-t\right) ^{s}dt \\
&&+f\left( b\right) g\left( a\right) \int_{0}^{1}t^{s}\left( 1-t\right)
dt+f\left( b\right) g\left( b\right) \int_{0}^{1}\left( 1-t\right) ^{s+1}dt
\end{eqnarray*}%
By substituting $ta+\left( 1-t\right) b=x,$ $\left( a-b\right) dt=dx$, we
obtain%
\begin{eqnarray*}
&&f\left( a\right) \int_{0}^{1}tg\left( ta+\left( 1-t\right) b\right)
dt+f\left( b\right) \int_{0}^{1}\left( 1-t\right) g\left( ta+\left(
1-t\right) b\right) dt \\
&&+g\left( a\right) \int_{0}^{1}t^{s}f\left( ta+\left( 1-t\right) b\right)
dt+g\left( b\right) \int_{0}^{1}\left( 1-t\right) ^{s}f\left( ta+\left(
1-t\right) b\right) dt \\
&=&\frac{f\left( a\right) }{\left( b-a\right) ^{2}}\int_{a}^{b}\left(
b-x\right) g\left( x\right) dx+\frac{f\left( b\right) }{\left( b-a\right)
^{2}}\int_{a}^{b}\left( x-a\right) g\left( x\right) dx \\
&&+\frac{g\left( a\right) }{\left( b-a\right) ^{s+1}}\int_{a}^{b}\left(
b-x\right) ^{s}f\left( x\right) dx+\frac{g\left( b\right) }{\left(
b-a\right) ^{s+1}}\int_{a}^{b}\left( x-a\right) ^{s}f\left( x\right) dx \\
&\leq &\int_{0}^{1}f\left( ta+\left( 1-t\right) b\right) g\left( ta+\left(
1-t\right) b\right) dt \\
&&+f\left( a\right) g\left( a\right) \int_{0}^{1}t^{s+1}dt+f\left( a\right)
g\left( b\right) \int_{0}^{1}t\left( 1-t\right) ^{s}dt \\
&&+f\left( b\right) g\left( a\right) \int_{0}^{1}t^{s}\left( 1-t\right)
dt+f\left( b\right) g\left( b\right) \int_{0}^{1}\left( 1-t\right) ^{s+1}dt
\\
&=&\frac{1}{b-a}\int_{a}^{b}f\left( x\right) g\left( x\right) dx+\frac{%
f\left( a\right) g\left( a\right) +f\left( b\right) g\left( b\right) }{s+2}
\\
&&+f\left( a\right) g\left( b\right) \beta \left( 2,s+1\right) +f\left(
b\right) g\left( a\right) \beta \left( s+1,2\right) \\
&=&\frac{1}{b-a}\int_{a}^{b}f\left( x\right) g\left( x\right) dx+\frac{%
M\left( a,b\right) }{s+2} \\
&&+f\left( a\right) g\left( b\right) \beta \left( 2,s+1\right) +f\left(
b\right) g\left( a\right) \beta \left( 2,s+1\right) \\
&=&\frac{1}{b-a}\int_{a}^{b}f\left( x\right) g\left( x\right) dx+\frac{%
M\left( a,b\right) }{s+2}+\beta \left( 2,s+1\right) \left[ f\left( a\right)
g\left( b\right) +f\left( b\right) g\left( a\right) \right] \\
&=&\frac{1}{b-a}\int_{a}^{b}f\left( x\right) g\left( x\right) dx+\frac{%
M\left( a,b\right) }{s+2}+\frac{\Gamma \left( 2\right) \Gamma \left(
s+1\right) }{\Gamma \left( s+3\right) }N\left( a,b\right) \\
&=&\frac{1}{b-a}\int_{a}^{b}f\left( x\right) g\left( x\right) dx+\frac{%
M\left( a,b\right) }{s+2}+\frac{\Gamma \left( s+1\right) }{\Gamma \left(
s+3\right) }N\left( a,b\right) \\
&=&\frac{1}{b-a}\int_{a}^{b}f\left( x\right) g\left( x\right) dx+\frac{%
M\left( a,b\right) }{s+2}+\frac{N\left( a,b\right) }{\left( s+1\right)
\left( s+2\right) }
\end{eqnarray*}

which completes the proof.\bigskip
\end{proof}

\begin{remark}
In Theorem \ref{t2}, if we choose $s=1,$ then (\ref{2.5}) reduces to (\ref%
{1.6})
\end{remark}

\bigskip

\begin{theorem}
\label{t3}Let $f,g:I\rightarrow 
\mathbb{R}
,$ $I\subset \left[ 0,\infty \right) $, $a,b\in I,$ $a<b$\ \ be functions
such that $f,g$ and $fg$ are in $L^{1}\left( \left[ a,b\right] \right) .$ If 
$f$ is $s_{1}-$convex and $g$ is $s_{2}-$convex in the second sense on $%
\left[ a,b\right] $ for some $s_{1},s_{2}\in \left( 0,1\right] $, then%
\begin{eqnarray}
&&\frac{f\left( a\right) }{\left( b-a\right) ^{s_{1}+1}}\int_{a}^{b}\left(
b-x\right) ^{s_{1}}g\left( x\right) dx+\frac{f\left( b\right) }{\left(
b-a\right) ^{s_{1}+1}}\int_{a}^{b}\left( x-a\right) ^{s_{1}}g\left( x\right)
dx  \notag \\
&&+\frac{g\left( a\right) }{\left( b-a\right) ^{s_{2}+1}}\int_{a}^{b}\left(
b-x\right) ^{s_{2}}f\left( x\right) dx+\frac{g\left( b\right) }{\left(
b-a\right) ^{s_{2}+1}}\int_{a}^{b}\left( x-a\right) ^{s_{2}}f\left( x\right)
dx  \label{2.6} \\
&\leq &\frac{1}{b-a}\int_{a}^{b}f\left( x\right) g\left( x\right) dx+\frac{1%
}{s_{1}+s_{2}+1}\left[ M\left( a,b\right) +s_{1}s_{2}\frac{\Gamma \left(
s_{1}\right) \Gamma \left( s_{2}\right) }{\Gamma \left( s_{1}+s_{2}+1\right) 
}N\left( a,b\right) \right] ,  \notag
\end{eqnarray}%
where $M\left( a,b\right) =f\left( a\right) g\left( a\right) +f\left(
b\right) g\left( b\right) $ and $N\left( a,b\right) =f\left( a\right)
g\left( b\right) +f\left( b\right) g\left( a\right) $.
\end{theorem}

\begin{proof}
Since $f$ is $s_{1}-$convex and $g$ is $s_{2}-$convex on $\left[ a,b\right]
, $ we have%
\begin{eqnarray*}
f\left( ta+\left( 1-t\right) b\right) &\leq &t^{s_{1}}f\left( a\right)
+\left( 1-t\right) ^{s_{1}}f\left( b\right) \\
g\left( ta+\left( 1-t\right) b\right) &\leq &t^{s_{2}}g\left( a\right)
+\left( 1-t\right) ^{s_{2}}g\left( b\right)
\end{eqnarray*}%
for all $a,b\in I$ and $t\in \left[ 0,1\right] .$ Now, using the elementary
inequality [\ref{ort}, p.4] $\left( a-b\right) \left( c-d\right) \geq 0$ $%
\left( a,b,c,d\in 
\mathbb{R}
\text{ and }a<b,c<d\right) ,$ we get inequality:%
\begin{eqnarray*}
&&t^{s_{1}}f\left( a\right) g\left( ta+\left( 1-t\right) b\right) +\left(
1-t\right) ^{s_{1}}f\left( b\right) g\left( ta+\left( 1-t\right) b\right) \\
&&+t^{s_{2}}g\left( a\right) f\left( ta+\left( 1-t\right) b\right) +\left(
1-t\right) ^{s_{2}}g\left( b\right) f\left( ta+\left( 1-t\right) b\right) \\
&\leq &f\left( ta+\left( 1-t\right) b\right) g\left( ta+\left( 1-t\right)
b\right) +t^{s_{1}+s_{2}}f\left( a\right) g\left( a\right) \\
&&+t^{s_{1}}\left( 1-t\right) ^{s_{2}}f\left( a\right) g\left( b\right)
+t^{s_{2}}\left( 1-t\right) ^{s_{1}}f\left( b\right) g\left( a\right) \\
&&+\left( 1-t\right) ^{s_{1}+s_{2}}f\left( b\right) g\left( b\right)
\end{eqnarray*}%
Integrating both sides of the above inequality over $\left[ 0,1\right] $, we
deduce that:%
\begin{eqnarray*}
&&f\left( a\right) \int_{0}^{1}t^{s_{1}}g\left( ta+\left( 1-t\right)
b\right) dt+f\left( b\right) \int_{0}^{1}\left( 1-t\right) ^{s_{1}}g\left(
ta+\left( 1-t\right) b\right) dt \\
&&+g\left( a\right) \int_{0}^{1}t^{s_{2}}f\left( ta+\left( 1-t\right)
b\right) dt+g\left( b\right) \int_{0}^{1}\left( 1-t\right) ^{s_{2}}f\left(
ta+\left( 1-t\right) b\right) dt \\
&\leq &\int_{0}^{1}f\left( ta+\left( 1-t\right) b\right) g\left( ta+\left(
1-t\right) b\right) dt \\
&&+f\left( a\right) g\left( a\right) \int_{0}^{1}t^{s_{1}+s_{2}}dt+f\left(
a\right) g\left( b\right) \int_{0}^{1}t^{s_{1}}\left( 1-t\right) ^{s_{2}}dt
\\
&&+f\left( b\right) g\left( a\right) \int_{0}^{1}t^{s_{2}}\left( 1-t\right)
^{s_{1}}dt+f\left( b\right) g\left( b\right) \int_{0}^{1}\left( 1-t\right)
^{s_{1}+s_{2}}dt
\end{eqnarray*}%
By substituting $ta+\left( 1-t\right) b=x,\left( a-b\right) dt=dx$, we
obtain 
\begin{eqnarray*}
&&f\left( a\right) \int_{0}^{1}t^{s_{1}}g\left( ta+\left( 1-t\right)
b\right) dt+f\left( b\right) \int_{0}^{1}\left( 1-t\right) ^{s_{1}}g\left(
ta+\left( 1-t\right) b\right) dt \\
&&+g\left( a\right) \int_{0}^{1}t^{s_{2}}f\left( ta+\left( 1-t\right)
b\right) dt+g\left( b\right) \int_{0}^{1}\left( 1-t\right) ^{s_{2}}f\left(
ta+\left( 1-t\right) b\right) dt \\
&=&\frac{f\left( a\right) }{\left( b-a\right) ^{s_{1}+1}}\int_{a}^{b}\left(
b-x\right) ^{s_{1}}g\left( x\right) dx+\frac{f\left( b\right) }{\left(
b-a\right) ^{s_{1}+1}}\int_{a}^{b}\left( x-a\right) ^{s_{1}}g\left( x\right)
dx \\
&&+\frac{g\left( a\right) }{\left( b-a\right) ^{s_{2}+1}}\int_{a}^{b}\left(
b-x\right) ^{s_{2}}f\left( x\right) dx+\frac{g\left( b\right) }{\left(
b-a\right) ^{s_{2}+1}}\int_{a}^{b}\left( x-a\right) ^{s_{2}}f\left( x\right)
dx \\
&\leq &\int_{0}^{1}f\left( ta+\left( 1-t\right) b\right) g\left( ta+\left(
1-t\right) b\right) dt \\
&&+f\left( a\right) g\left( a\right) \int_{0}^{1}t^{s_{1}+s_{2}}dt+f\left(
a\right) g\left( b\right) \int_{0}^{1}t^{s_{1}}\left( 1-t\right) ^{s_{2}}dt
\\
&&+f\left( b\right) g\left( a\right) \int_{0}^{1}t^{s_{2}}\left( 1-t\right)
^{s_{1}}dt+f\left( b\right) g\left( b\right) \int_{0}^{1}\left( 1-t\right)
^{s_{1}+s_{2}}dt \\
&=&\frac{1}{b-a}\int_{a}^{b}f\left( x\right) g\left( x\right) dx+\frac{%
f\left( a\right) g\left( a\right) +f\left( b\right) g\left( b\right) }{%
s_{1}+s_{2}+1} \\
&&+f\left( a\right) g\left( b\right) \beta \left( s_{1}+1,s_{2}+1\right)
+f\left( b\right) g\left( a\right) \beta \left( s_{2}+1,s_{1}+1\right) \\
&=&\frac{1}{b-a}\int_{a}^{b}f\left( x\right) g\left( x\right) dx+\frac{%
M\left( a,b\right) }{s_{1}+s_{2}+1} \\
&&+f\left( a\right) g\left( b\right) \beta \left( s_{1}+1,s_{2}+1\right)
+f\left( b\right) g\left( a\right) \beta \left( s_{1}+1,s_{2}+1\right) \\
&=&\frac{1}{b-a}\int_{a}^{b}f\left( x\right) g\left( x\right) dx+\frac{%
M\left( a,b\right) }{s_{1}+s_{2}+1} \\
&&+\beta \left( s_{1}+1,s_{2}+1\right) \left[ f\left( a\right) g\left(
b\right) +f\left( b\right) g\left( a\right) \right] \\
&=&\frac{1}{b-a}\int_{a}^{b}f\left( x\right) g\left( x\right) dx+\frac{%
M\left( a,b\right) }{s_{1}+s_{2}+1}+\frac{\Gamma \left( s_{1}+1\right)
\Gamma \left( s_{2}+1\right) }{\Gamma \left( s_{1}+s_{2}+2\right) }N\left(
a,b\right) \\
&=&\frac{1}{b-a}\int_{a}^{b}f\left( x\right) g\left( x\right) dx+\frac{1}{%
s_{1}+s_{2}+1}\left[ M\left( a,b\right) +s_{1}s_{2}\frac{\Gamma \left(
s_{1}\right) \Gamma \left( s_{2}\right) }{\Gamma \left( s_{1}+s_{2}+1\right) 
}N\left( a,b\right) \right]
\end{eqnarray*}%
which completes the proof.\bigskip
\end{proof}

\begin{remark}
In Theorem \ref{t3}, if we choose $s_{1}=s_{2}=1,$ then (\ref{2.6}) reduces
to (\ref{1.6})
\end{remark}

\bigskip

\begin{corollary}
With the above assumptions and under the conditions that $s_{1}=s_{2}=1$ and 
$x=\frac{a+b}{2},$ the following the inequality will be obtained%
\begin{eqnarray}
&&\frac{f\left( a\right) +f\left( b\right) }{2}g\left( \frac{a+b}{2}\right) +%
\frac{g\left( a\right) +g\left( b\right) }{2}f\left( \frac{a+b}{2}\right) 
\notag \\
&\leq &f\left( \frac{a+b}{2}\right) g\left( \frac{a+b}{2}\right) +\frac{%
M\left( a,b\right) }{3}+\frac{N\left( a,b\right) }{6}.  \label{2.9}
\end{eqnarray}
\end{corollary}

\begin{remark}
Similarly to Hadamard's inequality applications, some applications to
special means can be deduced by the above obtained two new theorems.
\end{remark}

\end{document}